\documentclass{amsart}
\usepackage[utf8]{inputenc}
\usepackage{amssymb,amsmath,amsthm,verbatim,color}
\usepackage[all]{xy}
  
\def\bP{\mathbb{P}}
\def\bb{\mathbb}
\def\c{\mathcal}
\def\wt{\widetilde}

\def\cO{\mathcal{O}}
\def\I{\mathcal I}
\def\log{\mathrm{log}}
\newtheorem{thm}{Theorem} 
\newtheorem{prop}[thm]{Proposition}
\newtheorem{lemma}[thm]{Lemma}
\newtheorem{cor}[thm]{Corollary}

\begin{document}
\title{The Wahl map of one-nodal curves on K3 surfaces}
\author{Edoardo Sernesi}
\date{}

\email{sernesi@mat.uniroma3.it}
\address{Dipartimento di Matematica e Fisica, Universit\`a Roma Tre, L.go S.L. Murialdo 1, 00146 Roma.}\date{}
\email{sernesi@gmail.com}
\thanks{We thank M. Halic and M. Kemeny for helpful  e-mail correspondence,  A. L. Knutsen and A. Bruno for  very useful comments on a preliminary version of this work, and the referee for his  careful  report. The author is member of GNSAGA-INDAM}
\subjclass[2010]{Primary 14J28, 14H10; Secondary 14H51}
 \dedicatory{This paper is dedicated to L. Ein on the occasion of his 60-th birthday.}

\keywords{K3 surface, Wahl map.}

 \begin{abstract}
 We consider a general primitively polarized K3 surface $(S,H)$ of genus $g+1$ and a  1-nodal curve $\wt C\in |H|$. We prove that the normalization $C$ of $\wt C$ has surjective Wahl map provided $g=40,42$ or $\ge 44$.
 \end{abstract}
 \maketitle

\section*{Introduction}\label{S:intro}

In this Note  we consider 1-nodal curves lying on a K3 surface and we study the gaussian map, or Wahl map, on their normalization. If we consider  a primitive linear system $|H|$ on a K3 surface $S$, then it is well known that every nonsingular $\wt C\in |H|$ has a non-surjective  Wahl map 
\[
\Phi_K: \bigwedge^2 H^0(\wt C,\omega_{\wt C}) \longrightarrow H^0(\wt C,\omega_{\wt C}^3)
\]
(see \S \ref{S:strategy} for the definition)   and that, if moreover  $\wt C\in |H|$ is general then it  is Brill-Noether-Petri general \cite{jW87,BM87,rL86}.  It is of some interest to decide whether the same properties hold for the normalization $C$ of a 1-nodal $\wt C\in |H|$, and more generally for the normalization of a singular $\wt C$ in $|H|$.  The Brill-Noether theory of singular curves on a K3 surface has received quite a lot of attention in recent times, see e.g. \cite{tG01,FKP07,BFT10,CK14,mK15}. On the other hand to our knowledge very little is known on their Wahl map. 
In \cite{mH15,mK15}  the authors consider a modified version of the Wahl map, which does not seem to have a direct and simple relation with the ordinary Wahl map $\Phi_K$; in particular their results  point towards the non-surjectivity of such modified map. A different point of view is taken in \cite{BF03}, where the authors give necessary conditions for a singular curve to be hyperplane section of a smooth surface, again in terms of non-surjectivity of certain maps.  On the other hand in \cite{FKPS08} it is proved that  the normalization $C$ (of genus 10) of a general 1-nodal curve $\wt C\in |H|$ on a general polarized $(S,H)$ of genus 11 has general moduli; then the main result of  \cite{CHM88} implies that 
$\Phi_K$ is surjective for such a curve.  This surjectivity result is extended in the present paper in  the following form:

\begin{thm}\label{T:main}
Let $(S,H)$ be a general primitively polarized K3 surface of genus $g+1$. Assume that $g= 40, 42$ or $\ge 44$.     Let $\wt C\in |H|$ be a  1-nodal curve and $C$ its normalization.  Then the Wahl map
\[
\Phi_K: \bigwedge^2 H^0(C,\omega_C) \longrightarrow H^0(C,\omega_C^3)
\]
is surjective.
\end{thm}

This   of course gives  another proof of the main result of \cite{CHM88} for the values of $g$ as in the statement, since 1-nodal curves are known to exist in $|H|$ for a general primitively polarized $(S,H)$ of any genus $g+1 \ge 2$ \cite{MM83,xC16}. 

Now a few words about the method of proof. Letting $P\in S$ be the unique singular point of $\wt C$
we consider the blow-up $\sigma:X:=\mathrm{Bl}_PS \longrightarrow S$   at $P$ and we let $E\subset X$ be the exceptional curve. Then the normalization of $\wt C$ is the strict transform     $C=\sigma^*\wt C-2E\subset X$. The Wahl map $\Phi_K$ on $C$ can be decomposed as
\[
\Phi_K = H^0(\rho)\cdot \Phi_{K_X+C}
\]
where $\Phi_{K_X+C}$ is a gaussian map on $X$ and $H^0(\rho)$ is induced in cohomology by a restriction homomorphism:
\[
\rho: \Omega^1_X(2K_X+2C) \longrightarrow \omega_C^3
\]
on $X$.
  We study  these two maps and prove their surjectivity separately.  This method of proof is analogous to the one adopted in the work of several authors before, notably \cite{BM87,CLM96,CLM00,DM92,jW90}. The restriction on the genus depends on the proof: one would expect the result to hold for $g=10$ (as it does indeed, as already remarked) and for $g\ge 12$. In fact  the surjectivity of $H^0(\rho)$ holds for   $g=10$ or $g \ge 12$ (Lemma \ref{L:main0}).   On the other hand the proof of the surjectivity of $\Phi_{K_X+C}$, which consists in adapting an analogous proof given in \cite{CLM00} for plane curves, leads to the restrictions on $g$ in Theorem \ref{T:main}: indeed this proof requires that we decompose a certain divisor on $X$ as the sum of three very ample ones and this decomposition forces the genus to increase.  
  
   Recent work by M. Kemeny \cite{mK15} implies that the curves $C$ considered here, i.e. normalizations of 1-nodal curves on a general primitive K3 surface, are generically Brill-Noether-Petri general and fill a locus in $\c M_g$, the coarse moduli space of curves of genus $g$, whose closure has dimension $19+g$. Theorem \ref{T:main} and \cite{jW87} imply that this naturally defined locus  is not contained in the closure of the so-called K3-locus (i.e. the locus of smooth curves that can be embedded in a K3 surface). 
  
  One can ask whether a result analogous to Theorem \ref{T:main} can be proved for the normalization of curves on K3 surfaces having a more complicated singular point. We did not consider this case. Note though that, to our knowledge, such curves are known to exist only in the case of $A_k$-singularities or ordinary triple points (see \cite{GK14,cG12}). 
  
  The paper is organized as follows. 
In \S \ref{S:strategy} we introduce the gaussian maps and explain the strategy of proof of the surjectivity of the Wahl map of a curve lying in a regular surface.
In \S \ref{S:lemma} we 
    prove  the surjectivity of $H^0(\rho)$  and in \S \ref{S:surface} we prove the surjectivity of $\Phi_{K_X+C}$. 
We work over $\bb C$.

  
  \section{Generalities on Gaussian maps}\label{S:strategy}

    In this section we   recall a few definitions and   basic facts concerning gaussian maps. 
  Given line bundles $L,M$ on a nonsingular projective variety $Y$ we consider:
$$
\mathcal{R}(L,M) = \ker[H^0(Y,L)\otimes H^0(Y,M) \to H^0(Y,L\otimes M)]
$$
Then we have a canonical map:
$$
\Phi_{L,M}: \mathcal{R}(L,M) \longrightarrow H^0(Y,\Omega^1_Y\otimes L\otimes M)
$$
called the \emph{gaussian map}, or \emph{Wahl map}, of $L,M$, which is   defined as follows. 
Let  $\Delta \subset Y\times Y$ be the diagonal and $p_1,p_2: Y\times Y \to Y$ the projections. Then 
$$
\mathcal{R}(L,M) = H^0(Y\times Y, p_1^*L\otimes p^*_2M\otimes \I_\Delta)
$$
Since $\I_\Delta\otimes \cO_\Delta= \Omega^1_Y$, the restriction to $\Delta$:
$$
p_1^*L\otimes p^*_2M\otimes \I_\Delta \longrightarrow p_1^*L\otimes p^*_2M\otimes \I_\Delta\otimes \cO_\Delta
$$
induces $\Phi_{L,M}$ on global sections.  The exact sequence:
\[
\xymatrix{
0\ar[r]&p_1^*L\otimes p^*_2M\otimes \I_\Delta^2\ar[r]&p_1^*L\otimes p^*_2M\otimes \I_\Delta\ar[r]&
p_1^*L\otimes p^*_2M\otimes \I_\Delta\otimes \cO_\Delta\ar[r]&0}
\]
shows that the vanishing:
\[
H^1(Y\times Y,p_1^*L\otimes p^*_2M\otimes \I_\Delta^2)=0
\]
is a sufficient condition for the surjectivity of  $\Phi_{L,M}$.

In case $L=M$ we have $\mathcal{R}(L,L)= I_2(Y)\oplus \bigwedge^2H^0(Y,L)$, where
$$
I_2(Y)= \ker[S^2H^0(Y,L) \to H^0(Y,L^2)]
$$
and $\Phi_{L,L}$ is zero on $I_2(Y)$. Therefore $\Phi_{L,L}$ is equivalent to its restriction  to $\bigwedge^2H^0(Y,L)$, which is denoted by
$$
\Phi_L:\bigwedge^2H^0(Y,L) \longrightarrow H^0(Y,\Omega^1_Y\otimes L^2)
$$
In particular, for a  non-hyperelliptic curve $C$  we are interested in  $\Phi_{K,K}$ or rather in 
\[
\Phi_K: \bigwedge^2H^0(C,\omega_C) \longrightarrow H^0(C,\omega_C^3)
\]
where $\cO_C(K)=\omega_C$ is the canonical invertible sheaf.

Suppose that $C \subset X$ where $X$ is a nonsingular regular surface. Then 
the exact sequence:
$$
\xymatrix{
0\to \cO_X(K_X)\ar[r]&\cO_X(K_X+C)\ar[r]&\omega_C\to 0
}
$$
shows that
$H^0(X,K_X+C)\to H^0(C,\omega_C)$ is  surjective. Moreover it is easy to show that $\Phi_K$ fits in the commutative diagram:
$$
\xymatrix{
\bigwedge^2H^0(X,K_X+C)\ar[d]\ar[r]^-{\Phi_{K_X+C}}& H^0(X,\Omega^1_X(2K_X+2C))\ar[d]^-{H^0(\rho)}\\
\bigwedge^2H^0(C,\omega_C)\ar[r]^-{\Phi_K}&H^0(\omega_C^3)
}
$$
where 
$$
\rho:  \Omega^1_X(2K_X+2C) \to \omega_C^3
$$
is the restriction map.   Since the left vertical map is surjective we have:

\begin{lemma}\label{L:surjcond}
In the above situation  
$$
\mathrm{Im}(\Phi_K) \subset \mathrm{Im}(H^0(\rho))
$$
and equality holds  if $\Phi_{K_X+C}$ is surjective. In particular, if both $\Phi_{K_X+C}$ and $H^0(\rho)$ are surjective, so is $\Phi_K$.
\end{lemma}


\section{The surjectivity of $H^0(\rho)$}\label{S:lemma}

As in the Introduction, we let $(S,H)$ be a   K3 surface with a   polarization of genus $g+1\ge 3$ and let $\wt C \in |H|$ be a curve with one node, i.e. an ordinary double point, at $P\in S$ and no other singularities. Consider the blow-up $\sigma:X:=\mathrm{Bl}_PS \longrightarrow S$ of $S$ at $P$, let $E\subset X$ be the exceptional curve and   $C=\sigma^*\wt C-2E\subset X$ the strict transform of $\wt C$. We have an    exact sequence on $X$:
\[
\xymatrix{
0\ar[r]&\Omega^1_X(\log\ C)(-C) \ar[r]&\Omega^1_X \ar[r] & \omega_C\ar[r]&0
}
\]
where $\Omega^1_X(\log\ C)$ is the sheaf of 1-forms with logarithmic poles along $C$ \cite{EV92}.
Tensoring with $\cO_X(2K_X+2C)$ we obtain:
\begin{equation}\label{E:logC1}
\xymatrix{
0\ar[r]&\Omega^1_X(\log\ C)(2K_X+C) \ar[r]&\Omega^1_X(2K_X+2C) \ar[r]^-\rho & \omega_C^3\ar[r]&0
}
\end{equation}

\begin{lemma}\label{L:main0}
Suppose that  $(S,H)$ is a general primitively polarized K3 surface of genus $g+1$, with $g=10$ or $g\ge 12$. Then,  with the same notations as above, we have:
\begin{align*}
h^1(X,\Omega^1_X(\log\ C)(2K_X+C))&=0
 \end{align*}
In particular  $H^0(\rho)$ is surjective.
\end{lemma}

\begin{proof}
The last assertion follows  from the exact sequence \eqref{E:logC1}.
Since $K_X=E$ we have $\cO_X(2K_X+C)=\sigma^*H$.
  Consider the relative cotangent sequence of $\sigma$: 
  \[
\xymatrix{
0\ar[r]&\sigma^*\Omega^1_S\ar[r]& \Omega^1_X\ar[r]^-\eta & \omega_E\ar[r]&0}
\]
 and tensor it    by $\sigma^*H$:
\[
\xymatrix{
0\ar[r]&\sigma^*\Omega^1_S(H)\ar[r]& \Omega^1_X(\sigma^*H)\ar[r]^-{\eta_H} & \omega_E\ar[r]&0}
\]
We have:
\[
H^2(X,\sigma^*\Omega^1_S(H))=H^2(S,\Omega^1_S(H))=H^0(S,T_S(-H))^\vee=0
\]
From the   assumption about the genus and from the generality of $(S,H)$ it follows that we also have:
\[
H^1(X,\sigma^*\Omega^1_S(H))=H^1(S,\Omega^1_S(H))=0
\]
(see \cite{aB04}, 5.2).  Therefore $\eta_H$ induces an isomorphism:
\[
H^1(\eta_H):H^1(X,\Omega^1_X(\sigma^*H))\cong H^1(E,\omega_E)\cong \bb C
\]
Then  in order to prove the lemma it suffices to show that in the exact sequence:
\[
\xymatrix{
0\ar[r]&\Omega^1_X(\sigma^*H) \ar[r]&\Omega^1_X(\log\ C)\otimes\sigma^*H\ar[r]&\sigma^*H\otimes\cO_C\ar[r]\ar@{=}[d]&0\\
&&&\omega_C(E)
}
\]
 the coboundary map:
\begin{equation}\label{E:tobeproved2}
\partial_H:H^0(C,\omega_C(E)) \longrightarrow H^1(X,\Omega^1_X(\sigma^*H))
\end{equation}
is non-zero. The above sequence is part of the following exact and commutative diagram:
$$
\xymatrix{
&\sigma^*\Omega^1_S(H)\ar[d]\ar@{=}[r]&\sigma^*\Omega^1_S(H)\ar[d]\\
0\ar[r]&\Omega^1_X(\sigma^*H) \ar[d]^-{\eta_H}\ar[r]&\Omega^1_X(\log\ C)\otimes\sigma^*H\ar[r]\ar[d]&\omega_C(E)\ar[r]\ar@{=}[d]&0\\
0\ar[r]&\omega_E\ar[r]\ar[d]&\mathcal{L}\ar[r]\ar[d]&\omega_C(E)\ar[r]\ar[d]&0 \\
&0&0&0
}
$$
where $\mathcal{L}$ is an invertible sheaf on $C+E$.  Since $H^1(\eta_H)$ is an isomorphism it suffices to show that 
the coboundary map $H^0(\omega_C(E)) \to H^1(\omega_E)$ of the last row is  non-zero or, equivalently, that $H^1(\mathcal{L})=0$.  Observe that $\mathcal{L}=\mathcal{M}\otimes\sigma^*H$, where $\mathcal{M}$ is defined by the following diagram:
$$
\xymatrix{
&\sigma^*\Omega^1_S\ar[d]\ar@{=}[r]&\sigma^*\Omega^1_S\ar[d]\\
0\ar[r]&\Omega^1_X \ar[d]^-{\eta}\ar[r]&\Omega^1_X(\log\ C)\ar[r]\ar[d]&\cO_C\ar[r]\ar@{=}[d]&0\\
0\ar[r]&\omega_E\ar[r]\ar[d]&\mathcal{M}\ar[r]\ar[d]&\cO_C\ar[r]\ar[d]&0 \\
&0&0&0
}
$$
It suffices to show that $\mathcal{M}\not\cong \cO_{C+E}$ because this will imply that $\mathcal{L}\not\cong \omega_{C+E}$, and in turn that $H^1(\mathcal{L})=0$.
The coboundary map of the middle row
$$
\partial: H^0(C,\cO_C) \longrightarrow H^1(X,\Omega^1_X)
$$
associates to $1\in H^0(C,\cO_C)$ the  Atiyah-Chern class of $\cO_X(C)$ and $H^1(\eta)(\partial(1))$ is its restriction to $E$.  Moreover
$$
\mathrm{ker}(H^1(\eta))\cong H^1(X,\sigma^*\Omega^1_S)
$$
is generated by   the Atiyah-Chern classes of the total trasforms under $\sigma$ of curves in $S$, which are trivial when restricted to $E$.  Since  $E\cdot C = 2$ we see that $\partial(1)\notin \mathrm{ker}(H^1(\eta))$. It follows that the coboundary of the last row:
$$
H^0(C,\cO_C) \longrightarrow H^1(E,\omega_E)
$$
is non-zero. Hence $\mathcal{M}\not\cong \cO_{C+E}$.
\end{proof}

  
  \section{The gaussian map on $X$}\label{S:surface}

  We keep the notations of  \S \ref{S:lemma}. We will  prove the following:

  \begin{prop}\label{P:main}
  Suppose that  $(S,H)$ is a general primitively polarized K3 surface of genus $g+1$, with   $g=40, 42$ or   $\ge 44$; let $\wt C\in |H|$ be a  1-nodal curve and $C$ its normalization. Let $X=\mathrm{Bl}_P(S)$ where $P\in \wt C$ is the node. Then the gaussian map
  \[
  \Phi_{K_X+C}: \bigwedge^2 H^0(X,\cO_X(K_X+C))\longrightarrow H^0(X,\Omega^1_X(2K_X+2C))
\]
is surjective.  
  \end{prop}
  
  We will need the following result and its corollary:
  
\begin{prop}\label{P:K3andreas} a) For every $d\ge 10$ there is a K3 surface $S$ containing two very ample nonsingular curves $A,B$ such that   $\mathrm{Pic}(S)= \bb Z [A]\oplus \bb Z [B] $, with intersection matrix:
  \[
  \begin{pmatrix}A^2&A\cdot B \\ B\cdot A&B^2\end{pmatrix} = \begin{pmatrix}8&d\\d&8\end{pmatrix}
  \]
  and $A$ and $B$  non-trigonal.

  b) For every $d\ge 12$ there is a K3 surface $S$ containing two very ample nonsingular curves $A,B$ such that   $\mathrm{Pic}(S)= \bb Z [A]\oplus \bb Z [B] $, with intersection matrix:
  \[
  \begin{pmatrix}A^2&A\cdot B \\ B\cdot A&B^2\end{pmatrix} = \begin{pmatrix}10&d\\d&8\end{pmatrix}
  \]
  and $A$ and $B$  non-trigonal.
  
  In both cases (a) and (b) the surface $S$ does not contain rational nonsingular curves $R$ such that $A\cdot R=1$ or $B\cdot R=1$. 
\end{prop}   

\begin{proof}
The Proposition is a special case of \cite{aK02}, Theorem 4.6. We obtain the Proposition by taking (with the notations used there) $(n,g-1)=(4,4)$ and $(n,g-1)=(5,4)$ respectively.  The restriction on $d$ is forced by the requirement that the hypotheses of the theorem apply symmetrically  w.r. to $A$ an $B$ so that both are very ample. The non-trigonality follows from the fact that $S$ is embedded by both $|A|$ and $|B|$ so to be an intersection of quadrics. The last assertion is proved as done in loc. cit. for $(-2)$ curves, by comparing discriminants.
\end{proof}

\begin{cor}\label{C:K3andreas}
If $S$ is as in Prop. \ref{P:K3andreas}(a) then $H=A+2B$ defines a primitive very ample divisor class of genus $g+1=21+2d$ for every $d \ge 10$. If $S$ is as in Prop. \ref{P:K3andreas}(b) then $H=A+2B$ defines a primitive very ample divisor class of genus $g+1=22+2d$ for every $d \ge 12$.
\end{cor}
\begin{proof}
Clearly $H$ is very ample and it is primitive because  the generator $A$ of $\mathrm{Pic}(S)$ appears with coefficient 1. The genus in either case is readily computed using the intersection matrix.
\end{proof}

\begin{proof}\emph{of Proposition \ref{P:main}.} By semicontinuity it suffices to prove the Proposition for just one primitively polarized K3 surface  for each value of $g$ as in the statement. We  take $(S,H)$ as in Corollary \ref{C:K3andreas}, distinguishing cases (a) and (b) according to the parity of $g$. Letting $\wt C,C,P$ and $X$ as in the statement, consider the product $X\times X$ and the blow-up $\pi: Y=\mathrm{Bl}_\Delta (X\times X) \to X\times X$ along the diagonal $\Delta$. Let $\Lambda \subset Y$ be the exceptional divisor. For any coherent sheaf $\mathcal F$ on $X$ we define $\mathcal F_i=(p_i\cdot \pi)^*\mathcal F$,  $i=1,2$. It suffices to prove that 
  \[
  H^1(X\times X,p_1^*\cO_X(K_X+C)\otimes p_2^*\cO_X(K_X+C)\otimes \I_\Delta^2)=0
  \]
  which is equivalent to:
  \[
  H^1(Y,(K_X+C)_1+(K_X+C)_2-2\Lambda)=0
  \]
 Note that we have:
 \[
 (K_X+C)_1+(K_X+C)_2-2\Lambda=K_Y+ C_1+ C_2 -3\Lambda
 \]
 and therefore we want to prove that:
 \begin{equation}\label{E:KW1}
 H^1(Y,\mathcal L)=0
 \end{equation}
 where we set $\mathcal L = \cO_Y(K_Y+C_1+C_2-3\Lambda)$. The proof is an adaptation of the proofs of Lemmas (3.1) and (3.10) of \cite{CLM00}. 
One uses the following:
  
  \begin{lemma}\label{L:product}
  Assume that $D$ is a very ample divisor on $X$. Then $D_1+ D_2-\Lambda$ is big and nef on $Y$.
  \end{lemma}
  \begin{proof}  See \cite{BEL91}, Claim 3.3. \end{proof}
  
  Let  $M:=\sigma^*H(-3E)$ on $X$. We have:
$$
M = \sigma^*A(-E)+2(\sigma^*B(-E))
$$
Since $A$ and $B$ are non-trigonal $S$ has no trisecant lines through $P$ and does not contain a line through $P$ whether it is embedded by $|A|$ or by $|B|$ (Proposition \ref{P:K3andreas}). Therefore every curvilinear subscheme of $S$ of length 3 containing $P$ imposes independent conditions to both $|A|$ and $|B|$. Then it follows from  \cite{mC02}, Prop. 1.3.4,  that both $\sigma^*A(-E)$ and $\sigma^*B(-E)$ are very ample on $X$.  Therefore by the Lemma we have that both $\sigma^*A(-E)_1+ \sigma^*A(-E)_2-\Lambda$ and $\sigma^*B(-E)_1+ \sigma^*B(-E)_2-\Lambda$  are big and nef. It follows that the divisor $M_1+M_2-3\Lambda$ is big and nef, being the sum of three big and nef divisors.  

Since $C\sim M+E$ we have $C_1+C_2\sim M_1+M_2+E_1+E_2$ and therefore we have an exact sequence on $Y$:
$$
0\to\cO_Y(K_Y+ M_1+ M_2-3\Lambda)\to\c L\to\cO_{E_1+E_2}(\c L)\to 0
$$
By Kawamata-Vieweg we have $H^1(Y,K_Y+ M_1+ M_2-3\Lambda) =0$: therefore in order to prove \eqref{E:KW1} it suffices to show that
\begin{equation}\label{E:KW2}
H^1(E_1+E_2,\cO_{E_1+E_2}(\c L))=0
\end{equation}
Letting $W:= E_1\cap E_2$ we have an exact sequence:
$$
0\to \cO_{E_1}(\mathcal L-W) \to \cO_{E_1+E_2}(\mathcal L) \to \cO_{E_2}(\mathcal L)\to 0
$$
and, by symmetry,  it suffices to  prove that:
\begin{equation}\label{E:KW3}
H^1(E_1,\cO_{E_1}(\mathcal L-W))=0
\end{equation}
and
$$
H^1(E_1, \cO_{E_1}(\mathcal L))=0
$$
We can then consider the exact sequence on $E_1$:
$$
0\to \cO_{E_1}(\mathcal L-W)\to \cO_{E_1}(\mathcal L) \to \cO_W(\mathcal L)\to 0
$$
and finally we are reduced to prove \eqref{E:KW3}
and
\begin{equation}\label{E:KW4}
H^1(W, \cO_W(\mathcal L))=0
\end{equation}
Let $U\cong E\times E$ be the proper transform of $E\times E$ in $Y$.  Then in $E_1$ we have $W = U+\Lambda_{|E_1}$.
As in \cite{CLM00}, proof of Lemma (3.1), one shows that $\Lambda_{|E_1}\cong \bP\mathcal E$, where $\mathcal E= \cO_{\bP^1}(1)\oplus\cO_{\bP^1}(-2)$, and $\mathcal L_{|\bP\mathcal E} = \cO_{\bP\mathcal E}(2C_0+2f)$ where $C_0\in |\cO_{\bP\mathcal E}(1)|$ and $f$ is a fibre of $\bP\mathcal E\to \bP^1$. Now the proof proceeds  as in \cite{CLM00}, Lemma (3.1), after having proved  that $\cO_{E_1}(\mathcal L-W)=\cO_{E_1}(K_Y+ M_1+ M_2-3\Lambda)$ is big and nef. This last fact is obtained exactly as in the proof of Lemma (3.10) of \cite{CLM00}, using the fact that both $\sigma^*A(-E)$ and $\sigma^*B(-E)$ are very ample on $X$.
\end{proof}

\begin{proof}\emph{of Theorem \ref{T:main}.}
Recalling  Lemma  \ref{L:surjcond}, the theorem follows immediately  from Lemma \ref{L:main0} and from Proposition \ref{P:main}.
\end{proof}


 \end{document}